\documentclass{amsart}
\usepackage{amssymb}
\usepackage{color}
\usepackage{graphicx}
\usepackage{longtable}
\usepackage[font=small,labelfont=bf]{caption}

\theoremstyle{plain}
\newtheorem{theorem}{Theorem}[section]
\newtheorem{corollary}[theorem]{Corollary}

\newtheorem{proposition}[theorem]{Proposition}

\newtheorem{lemma}[theorem]{Lemma}
\numberwithin{equation}{section}

\makeatletter
\@namedef{subjclassname@2020}{
  \textup{2020} Mathematics Subject Classification}
\makeatother

\begin{document}

\title[Hankel Determinants of shifted sequences]{Hankel Determinants of shifted sequences of Bernoulli and Euler numbers}

\author{Karl Dilcher}
\address{Department of Mathematics and Statistics\\
         Dalhousie University\\
         Halifax, Nova Scotia, B3H 4R2, Canada}
\email{dilcher@mathstat.dal.ca}

\author[Lin Jiu]{Lin Jiu\textsuperscript{*}}
\thanks{*Corresponding author}
\address{Zu Chongzhi Center for Mathematics and Computational Sciences\\
Duke Kunshan University\\
Kunshan, Suzhou, Jiangsu Province, 215316, PR China}
\email{lin.jiu@dukekunshan.edu.cn}

\thanks{Research supported in part by the Natural Sciences and Engineering
        Research Council of Canada, Grant \# 145628481}

\date{}

\setcounter{equation}{0}

\begin{abstract}
Hankel determinants of sequences related to Bernoulli and Euler numbers have
been studied before, and numerous identities are known. However, when a 
sequence is shifted by one unit, the situation often changes significantly.
In this paper we use classical orthogonal polynomials and related methods to
prove a general result concerning Hankel determinants for shifted sequences.
We then apply this result to obtain new Hankel determinant evaluations for a 
total of $13$ sequences related to Bernoulli and Euler numbers, one of which 
concerns Euler polynomials.
\end{abstract}

\maketitle

\section{Introduction}

The {\it Hankel determinant} of a sequence ${\bf c}=(c_0,c_1,\ldots)$ of 
numbers or polynomials is defined as the determinant of the {\it Hankel matrix},
or {\it persymmetric matrix}, given by
\begin{equation}\label{1.0}
\big(c_{i+j}\big)_{0\leq i,j\leq n}
=\begin{pmatrix}
c_{0} & c_{1} & c_{2} & \cdots & c_{n}\\
c_{1} & c_{2} & c_{3} & \cdots & c_{n+1}\\
c_{2} & c_{3} & c_{4} & \cdots & c_{n+2}\\
\vdots & \vdots & \vdots & \ddots & \vdots\\
c_{n} & c_{n+1} & c_{n+2} & \cdots & c_{2n}
\end{pmatrix}.
\end{equation}
Hankel determinants of various classes of sequences have been extensively
studied, partly due to their close relationship with classical orthogonal 
polynomials; see, e.g., \cite[Ch.~2]{Is}. In fact, many evaluations of Hankel 
determinants come from this connection and a related connection with continued 
fractions. For numerous results see, e.g., the very extensive treatments in 
\cite{Kr1, Kr2, Mi}, and the numerous references provided there.

In the recent paper \cite{DJ1} we used the connection with orthogonal 
polynomials
and continued fractions to find new evaluations of Hankel determinants of
certain subsequences of Bernoulli and Euler polynomials. This was followed in
\cite{DJ2} by evaluations of Hankel determinants of various other sequences
related to Bernoulli and Euler numbers and polynomials. 

We recall that the {\it Bernoulli numbers} $B_n$
and {\it polynomials} $B_n(x)$ are usually defined by the generating functions
\begin{equation}\label{1.2}
\frac{t}{e^t-1} = \sum_{n=0}^\infty B_n\frac{t^n}{n!}\qquad\hbox{and}\qquad
\frac{te^{xt}}{e^t-1} = \sum_{n=0}^\infty B_n(x)\frac{t^n}{n!}.
\end{equation}
We have $B_0=1$, $B_1=-1/2$, and $B_{2j+1}=0$ for $j\geq 1$; a few further
values are listed in Table~\ref{tab:1}. The {\it Euler numbers} $E_n$ and 
{\it polynomials} $E_n(x)$ are defined by the generating functions
\begin{equation}\label{1.3}
\frac{2}{e^t+e^{-t}} = \sum_{n=0}^\infty E_n\frac{t^n}{n!}\qquad\hbox{and}\qquad
\frac{2e^{xt}}{e^t+1} = \sum_{n=0}^\infty E_n(x)\frac{t^n}{n!}.
\end{equation}
The first few values are again given in Table~\ref{tab:1}. Comparing the generating 
functions in \eqref{1.2} and in \eqref{1.3}, respectively, we get
\begin{equation}\label{1.3a}
B_n(x) = \sum_{j=0}^n\binom{n}{j}B_j x^{n-j},\qquad
2^nE_n(x) = \sum_{j=0}^n\binom{n}{j}E_j(2x-1)^{n-j},
\end{equation}
and thus, for all $n=0, 1,\ldots$,
\begin{equation}\label{1.3b}
B_n = B_n(0),\qquad E_n = 2^nE_n\big(\tfrac{1}{2}\big). 
\end{equation}
The four sequences in \eqref{1.2} and in \eqref{1.3} are
among the most important special number and polynomial sequences in mathematics,
with numerous applications in number theory, combinatorics, numerical analysis,
and other areas. Further properties can be found, e.g., in \cite[Ch.~24]{DLMF}.

\begin{table}
\begin{center}
\renewcommand*{\arraystretch}{1.2}
\begin{tabular}{|r||r|r|l|l|r|}
\hline
$n$ & $B_n$ & $E_n$ & $B_n(x)$ & $E_n(x)$ & $E_n(1)$\\
\hline
0 & 1 & 1 &  1 & 1 & 1\\
1 & $-1/2$ & 0 & $x-\tfrac{1}{2}$ & $x-\tfrac{1}{2}$ & $1/2$ \\
2 & $1/6$ & $-1$  & $x^2-x+\tfrac{1}{6}$ & $x^2-x$ & $0$ \\
3 & 0 & 0 & $x^3-\tfrac{3}{2}x^2+\tfrac{1}{2}x$ & $x^3-\tfrac{3}{2}x^2+\tfrac{1}{4}$ & $-1/4$ \\
4 & $-1/30$ & 5 & $x^4-2x^3+x^2-\tfrac{1}{30}$ & $x^4-2x^3+x$ & $0$ \\
5 & 0 & 0 & $x^5-\tfrac{5}{2}x^4+\tfrac{5}{3}x^3-\tfrac{1}{6}x$ &
$x^5-\tfrac{5}{2}x^4+\tfrac{5}{2}x^2-\tfrac{1}{2}$ & $1/2$  \\
6 & $1/42$ & $-61$ & $x^6-3x^5+\tfrac{5}{2}x^4-\tfrac{1}{2}x^2+\tfrac{1}{42}$ &
$x^6-3x^5+5x^3-3x$ & $0$\\
\hline
\end{tabular}
\end{center}
\captionsetup{font=large}
\caption{$B_n, E_n, B_n(x)$, $E_n(x)$, and $E_n(1)$ for $0\leq n\leq 6$. \label{tab:1}}
\end{table}

The paper \cite{DJ2} also contains a list of all ``Hankel-Bernoulli/Euler''
identities known to us. It turned out that the majority of such identities,
when written in a standard way, are of a very specific form. Indeed, if we set
\begin{equation}\label{1.4}
H_n({\bf c}) = H_n(c_k) = \det_{0\leq i,j\leq n}\big(c_{i+j}\big),
\end{equation}
then for a large number of sequences ${\bf c}$ related to Bernoulli and Euler
numbers and polynomials we have
\begin{equation}\label{1.5}
H_n(c_k) = (-1)^{\varepsilon(n)}\cdot a^{n+1}\cdot\prod_{\ell=1}^{n}b(\ell)^{n+1
-\ell},
\end{equation}
where $\varepsilon(n)$ is either $0$ or a polynomial in $n$ of degree at most $2$, 
$a$ is a positive rational number, and $b(\ell)$ is a rational function 
in $\ell$ having only linear factors in the numerator and the denominator. (In
a few cases where $c_k$ is a polynomial sequence, $a$ and $b(\ell)$ also have
some linear factors in $x$). For instance, if $c_k=B_{2k+2}$, then
\begin{equation}\label{1.6}
\varepsilon(n)=0,\quad a=\frac{1}{6},\quad 
b(\ell)=\frac{\ell^3(\ell+1)(2\ell-1)(2\ell+1)^3}{(4\ell-1)(4\ell+1)^2(4\ell+3)},
\end{equation}
and when $c_k=B_{2k+4}$, we have
\begin{equation}\label{1.7}
\varepsilon(n)=n+1,\quad a=\frac{1}{30},\quad 
b(\ell)=\frac{\ell(\ell+1)^3(2\ell+1)^3(2\ell+3)}{(4\ell+1)(4\ell+3)^2(4\ell+5)}.
\end{equation}
These two expressions were adapted from the identities (3.59) and (3.60),
respectively, in \cite{Kr1}.

On the other hand, while there is such a formula for $c_k=B_{2k}(\frac{1}{2})$,
no general identity has been known for the seemingly more natural sequence
$c_k=B_{2k}$. In fact, the four smallest nontrivial Hankel determinants 
$H_n(B_{2k})$, for $n=1, 2, 3, 4$, factor as
\[
\frac{-11}{2^2\cdot3^2\cdot 5},\quad\frac{137}{2\cdot3^2\cdot5^3\cdot 7^2},\quad
\frac{-2^2\cdot 3}{5\cdot 7^4\cdot 13},\quad
\frac{2^{10}\cdot 7129}{5\cdot 7^2\cdot 11^4\cdot 13^3\cdot 17},
\]
respectively. This indicates that we cannot expect an identity such as 
\eqref{1.5}. However,
further numerical experiments revealed that for all $n\geq 1$ we might
conjecture
\begin{equation}\label{1.8}
H_n(B_{2k})=(-1)^n\frac{(4n+3)!}{(n+1)\cdot(2n+1)!^3}\cdot{\mathcal H}_{2n+1}
\cdot H_n(B_{2k+2}),
\end{equation}
where $H_n(B_{2k+2})$ is given by \eqref{1.5} and \eqref{1.6}, and 
${\mathcal H}_n$ is the $n$th harmonic number
\begin{equation}\label{1.9}
{\mathcal H}_n = \sum_{j=1}^n\frac{1}{j}.
\end{equation}

It is the purpose of this paper to prove the identity \eqref{1.8} and a number
of other similar and apparently new identities. When $(b_k)_{k\geq 0}$ is a
given sequence, the identities are all of the form
\begin{equation}\label{1.10}
H_n(b_k)=F_nH_n(b_{k+1}),
\end{equation}
where $H_n(b_{k+1})$ has a known evaluation. The sequence $F_n$ is most often
a sequence of harmonic or related numbers, but factorials and in one case a 
recurrence sequence also occur.

This paper is structured as follows. In Section~\ref{sec2} we provide some necessary
background, mainly related to orthogonal polynomials, and we prove a lemma that
will be the basis for all further results. In Section~\ref{sec3} we state and prove six
different Hankel determinant evaluations that follow more or less directly from
this lemma. Then, in Section~\ref{sec4}, we state without proof several further known 
auxiliary results, which will then be used in the remaining three sections to
prove a few more Hankel determinant evaluations. The last of these sections is
different in that it deals with sequences of polynomials.

We conclude this introduction with a summary of Hankel determinant identities 
that are proved in this paper. See Table~\ref{tab:2}, were the sequence $(b_k)$ is as in 
\eqref{1.10}, and those values of $b_0$ that are consistent with the general
$b_k$ are given in parentheses. 

\begin{table}
\begin{center}
\renewcommand*{\arraystretch}{1.2}
\begin{tabular}{|l|c|c||l|c|c|}
\hline
$b_k, k\geq 1$ & $b_0$ & Prop. & $b_k, k\geq 1$ & $b_0$ & Prop. \\
\hline
$B_{k-1}$ & 0 & 3.1 & $E_{k+3}(1)$ & $(\tfrac{-1}{4})$ & 5.2 \\
$B_{2k}$ & (1) & 6.1 & $E_{2k-1}(1)$ & 0 & 3.3 \\
$(2k+1)B_{2k}$ & (1) & 6.2 & $E_{2k+5}(1)$ & $(\tfrac{1}{2})$ & 5.1 \\
$(2^{2k}-1)B_{2k}$ & (0) & 3.4 & $E_{k}(1)/k!$ & (1) & 3.6 \\
$(2k+1)E_{2k}$ & 0 & 3.5 & $E_{2k-1}(1)/(2k-1)!$ & 0 & 6.3 \\
$E_{2k-2}$ & 0 & 7.3 & $E_{2k-2}(\tfrac{x+1}{2})$ & 0 & 7.2 \\
$E_{k-1}(1)$ & 0 & 3.2 & & & \\
\hline
\end{tabular}
\end{center}
\captionsetup{font=large}
\caption{Summary of results. \label{tab:2}}
\end{table}

\section{Orthogonal polynomials and a fundamental lemma}\label{sec2}

We begin this section with some necessary background on the connection 
between orthogonal polynomials and Hankel determinants. 
All this is well-known and can also be found in concise form in 
\cite{DJ1} and \cite{DJ2}. We repeat this material here for easy reference,
and to make this paper self-contained. The second part of this section is new,
and will be the basis for much of what follows.

\subsection{Orthogonal polynomials}

Suppose we are given a sequence ${\bf c}=(c_0, c_1, \ldots)$ of numbers; then 
we can define a linear functional $L$ on polynomials by
\begin{equation}\label{2.1} 
L(x^k) = c_k,\quad k = 0, 1, 2, \ldots.
\end{equation}
We may also normalize the sequence such that $c_0=1$. We now summarize several 
well-known facts and state them as a lemma with two corollaries; see, e.g., 
\cite[Ch.~2]{Is} and \cite[pp.~7--10]{Chi}. 

\begin{lemma}\label{lem:2.1}
Let $L$ be the linear functional in \eqref{2.1}. If (and only if) 
$H_n(c_k)\neq 0$ for all $n=0, 1, 2, \ldots$, there exists a unique sequence of
monic polynomials $P_n(y)$ of degree $n$, $n=0, 1, \ldots$, and a sequence of 
positive numbers $(\zeta_n)_{n\geq 1}$, with $\zeta_0=1$, such that 
\begin{equation}\label{2.2}
L\left(P_m(y)P_n(y)\right) = \zeta_n\delta_{m,n},
\end{equation}
where $\delta_{m,n}$ is the Kronecker delta function. Furthermore, for all 
$n\geq 1$ we have $\zeta_n=H_n({\bf c})/H_{n-1}({\bf c})$, and for $n\geq 1$,
\begin{equation}\label{2.3}
P_n(y) = \frac{1}{H_{n-1}({\bf c})}\det
\begin{pmatrix}
c_{0} & c_{1} & \cdots & c_{n}\\
c_{1} & c_{2} & \cdots & c_{n+1}\\
\vdots & \vdots & \ddots & \vdots\\
c_{n-1} & c_{n} & \cdots & c_{2n-1} \\
1 & y & \cdots & y^n
\end{pmatrix},
\end{equation}
where the polynomials $P_n(y)$ satisfy the $3$-term recurrence relation 
$P_0(y)=1$, $P_1(y)=y+s_0$, and
\begin{equation}\label{2.4}
P_{n+1}(y) = (y+s_n)P_n(y) - t_nP_{n-1}(y)\qquad (n\geq 1),
\end{equation}
for some sequences $(s_n)_{n\geq 0}$ and $(t_n)_{n\geq 1}$.
\end{lemma}

We now multiply both sides of \eqref{2.3} by $y^r$ and replace $y^j$ by $c_j$, 
which includes replacing the constant term $1$ by $c_0$ for $r=0$.
Then for $0\leq r\leq n-1$ the last row of the matrix in \eqref{2.3} is 
identical with one of the previous rows, and thus the determinant is $0$. When 
$r=n$, the determinant is $H_n({\bf c})$. We therefore have the following 
result.

\begin{corollary}\label{cor:2.2}
With the sequence $(c_k)$ and the polynomials $P_n(y)$ as above, we have
\begin{equation}\label{2.5}
y^rP_n(y)\bigg|_{y^k=c_k} = \begin{cases}
0 &\hbox{when}\;\; 0\leq r\leq n-1,\\
H_n({\bf c})/H_{n-1}({\bf c}) &\hbox{when}\;\; r=n.
\end{cases}
\end{equation}
\end{corollary}

The polynomials $P_n(y)$ are known as ``the monic orthogonal polynomials 
belonging to the sequence ${\bf c}=(c_0, c_1, \ldots)$", or ``the polynomials 
orthogonal with respect to $\bf c$".
Another important consequence of Lemma~\ref{lem:2.1} is the main reason for the
specific form of the general formula \eqref{1.5}.

\begin{corollary}\label{cor:2.3}
With the sequence $(t_n)$ as in \eqref{2.4}, we have
\begin{equation}\label{2.6}
H_n({\bf c}) = t_1^nt_2^{n-1}\cdots t_{n-1}^2t_n\qquad (n\geq 0).
\end{equation}
\end{corollary}

The next lemma, which will also be required later in this paper, deals with the
case where $\bf c$ is a sequence of functions in a single variable $x$. It was
proved as Lemma~5.7 in \cite{DJ2}.

\begin{lemma}\label{lem:2.4}
Let $c_k(x)$ be a sequence of $C^1$ functions, and let
$P_n(y;x)$ be the corresponding monic orthogonal polynomials. If $c_k(x_0)=0$
for some $x_0\in{\mathbb C}$ and for all $k\geq 0$, then $P_n(y;x_0)$ are the 
monic orthogonal polynomials with respect to the sequence of derivatives
$c_k'(x_0)$, as long as $H_n(c_k'(x_0))$ are all nonzero.
\end{lemma}

\subsection{A fundamental lemma}
We are now ready to state and prove a general lemma which will be used in
most of the proofs that follow.

\begin{lemma}\label{lem:2.5}
Let ${\bf c}=(c_0, c_1, \ldots)$ be a sequence such that the unique sequence
$P_n(y)$, $n\geq 0$, of polynomials orthogonal with respect to $\bf c$ exists.
Let $\alpha$ be a constant and define the sequence ${\bf b}=(b_0, b_1, \ldots)$
by 
\begin{equation}\label{2.7}
b_k:=\begin{cases}
\alpha, &k=0,\\
c_{k-1}, &k\geq 1.
\end{cases}
\end{equation}
Then for all $n\geq 2$ we have
\begin{equation}\label{2.8}
\frac{H_{n+1}(b_k)}{H_n(c_k)} = -s_n\frac{H_n(b_k)}{H_{n-1}(c_k)}
-t_n\frac{H_{n-1}(b_k)}{H_{n-2}(c_k)},
\end{equation}
where the sequences $(s_n)$ and $(t_n)$ are as in \eqref{2.4}.
\end{lemma}

Some care must be taken when $\alpha=0$. In this case $H_0(b_k)=0$ and thus 
$P_1(y)$, as given by \eqref{2.3}, does not exist. 
However, as long as $H_n(b_k)\neq 0$ for $k\geq 1$, due to uniqueness the terms
$P_n(y)$, for $n\geq 2$, are still given by \eqref{2.3}; meanwhile, for $n=1$,
we can compute the Hankel determinant directly.

\begin{proof}[Proof of Lemma~\ref{lem:2.5}]
By the definitions \eqref{1.0}, \eqref{1.4}, and \eqref{2.7} we have
\begin{align*}
H_{n}(b_{k}) & =\det\begin{pmatrix}\alpha & c_{0} & c_{1} & \cdots & c_{n-1}\\
c_{0} & c_{1} & c_{2} & \cdots & c_{n}\\
\vdots & \vdots & \vdots & \ddots & \vdots\\
c_{n-1} & c_{n} & c_{n+1} & \cdots & c_{2n-1}
\end{pmatrix}\\
 & =(-1)^{n}\det\begin{pmatrix}c_{0} & c_{1} & c_{2} & \cdots & c_{n}\\
\vdots & \vdots & \vdots & \ddots & \vdots\\
c_{n-1} & c_{n} & c_{n+1} & \cdots & c_{2n-1}\\
\alpha & c_{0} & c_{1} & \cdots & c_{n-1}
\end{pmatrix}\\
 & =(-1)^{n}H_{n-1}(c_{k})\left(
\frac{P_{n}(y)-P_{n}(0)}{y}\bigg|_{y^{k}=c_{k}}+\alpha P_n(0)\right),
\end{align*}
where in the last equation we have used \eqref{2.3}. Hence
\begin{equation}\label{2.9}
\frac{P_{n}(y)-P_{n}(0)}{y}\bigg|_{y^{k}=c_{k}}+\alpha P_n(0)
= (-1)^n\frac{H_{n}(b_{k})}{H_{n-1}(c_{k})}.
\end{equation}
Next, setting $y=0$ in \eqref{2.4}, we get
\begin{equation}\label{2.10}
P_{n+1}(0) = s_nP_n(0)-t_nP_{n-1}(0).
\end{equation}
We then subtract \eqref{2.10} from \eqref{2.4}, divide both sides by $y$, and
subtract from the resulting equation once again \eqref{2.10}, multiplied by
$\alpha$. This gives
\begin{align}
\frac{P_{n+1}(y)-P_{n+1}(0)}{y}&-\alpha P_{n+1}(0) =
P_{n}(y)+s_{n}\left(\frac{P_{n}(y)-P_{n}(0)}{y}-\alpha P_{n}(0)\right)\label{2.11}\\
&\quad -t_{n}\left(\frac{P_{n-1}(y)-P_{n-1}(0)}{y}-\alpha P_{n-1}(0)\right).\nonumber
\end{align}
Finally, if we evaluate $y^k=c_k$, then by \eqref{2.5} the first term on the
right-hand side of \eqref{2.11}, namely $P_n(y)$, vanishes, and \eqref{2.11} 
with \eqref{2.9} yields the desired identity \eqref{2.8}.
\end{proof}

\section{Hankel determinant identities, I}\label{sec3}

In this section we deal with those identities whose proofs follow most directly
from Lemma~\ref{lem:2.5}. In addition to the harmonic numbers ${\mathcal H}_n$
in \eqref{1.9}, we require the following related sequences: For $n\geq 1$ we
denote
\begin{equation}\label{3.1}
{\mathcal H}_{1,n}^{-}:=\sum_{j=1}^n\frac{(-1)^{j-1}}{j},\qquad
{\mathcal H}_{2,n}^{-}:=\sum_{j=1}^n\frac{(-1)^{j-1}}{j^2}.
\end{equation}
We can now state and then prove the following six results.

\begin{proposition}\label{prop:3.1}
If the sequence $(b_0, b_1, \ldots)$ is defined by
\[
b_k:=\begin{cases}
0,& k=0,\\
B_{k-1}, & k\geq 1.
\end{cases}
\]
then
\begin{equation}\label{3.2}
H_n(b_k) = \frac{2\cdot(2n+1)!}{n!^3}\cdot{\mathcal H}_{2,n}^{-}
\cdot H_n(B_k).
\end{equation}
\end{proposition}

The known Hankel determinant on the right of \eqref{3.2}, and in all further
results, will be given in the relevant proofs below. 

\begin{proposition}\label{prop:3.2}
If the sequence $(b_0, b_1, \ldots)$ is defined by
\[
b_k:=\begin{cases}
0,& k=0,\\
E_{k-1}(1),& k\geq 1,
\end{cases}
\]
then
\begin{equation}\label{3.3}
H_n(b_k) = (-1)^{n-1}\frac{2^{n+1}}{n!}\cdot{\mathcal H}_{1,n}^{-}
\cdot H_n\big(E_k(1)\big).
\end{equation}
\end{proposition}

Here it should be mentioned that
\begin{equation}\label{3.4}
E_k(1) = \frac{2}{k+1}\left(2^{k+1}-1\right)B_{k+1}\quad (k\geq 1);
\end{equation}
see, e.g., \cite[Eq.~24.4.26]{DLMF}. The following result can be seen as the 
odd-index analogue of Proposition~\ref{prop:3.2}.

\begin{proposition}\label{prop:3.3}
If the sequence $(b_0, b_1, \ldots)$ is defined by
\[
b_k:=\begin{cases}
0,& k=0,\\
E_{2k-1}(1),& k\geq 1,
\end{cases}
\]
then
\begin{equation}\label{3.5}
H_n(b_k) = (-1)^n\frac{2^{2n+1}}{(2n+1)!}\cdot{\mathcal H}_n
\cdot H_n\big(E_{2k+1}(1)\big).
\end{equation}
\end{proposition}

Once again, \eqref{3.4} could be used to rewrite this result in terms of
Bernoulli numbers. The next result is related to the 
right-hand side of \eqref{3.4}.

\begin{proposition}\label{prop:3.4}
For all $n\geq 1$ we have
\begin{equation}\label{3.6}
H_n\big((2^{2k}-1)B_{2k}\big) = \frac{(-1)^n}{n!(n+1)!}\cdot{\mathcal H}_n
\cdot H_n\big((2^{2k+2}-1)B_{2k+2}\big).
\end{equation}
\end{proposition}

The following result is somewhat different from the previous ones.
While its proof is similar to that of Proposition~\ref{prop:3.4}, the 
statement involves a recurrence sequence in place of a harmonic or related
sequence.

\begin{proposition}\label{prop:3.5}
If the sequence $(b_0, b_1, \ldots)$ is defined by
\[
b_k:=\begin{cases}
0,& k=0,\\
(2k+1)E_{2k},& k\geq 1,
\end{cases}
\]
then for all $n\geq 1$,
\begin{equation}\label{3.6a}
H_n(b_k) = \frac{(-1)^n}{(2^nn!)^4}\cdot h_n\cdot H_n\big((2k+1)E_{2k}\big),
\end{equation}
where the sequence $(h_n)_{n\geq 0}$ is defined by $h_0=0, h_1=1$, and for
$n\geq 1$
\begin{equation}\label{3.6b}
h_{n+1} = (8n^2+8n+3)h_n - (2n)^4h_{n-1}.
\end{equation}
\end{proposition}

The first few terms of the sequence $(h_n)$, starting with $n=1$, are 
\[
1,\; 19,\; 713,\; 45\,963,\; 4\,571\,521,\; 651\,249\,603,\; 125\,978\,555\,961.
\]

The final result in this section is again different from the previous ones in
that it contains neither a harmonic-type sequence, nor a recurrence sequence.

\begin{proposition}\label{prop:3.6}
For all $n\geq 1$ we have
\begin{equation}\label{3.6c}
H_n\left(\frac{E_k(1)}{k!}\right)  = \frac{(2n+2)!}{(n+1)!}\cdot 
H_n\left(\frac{E_{k+1}(1)}{(k+1)!}\right).
\end{equation}
\end{proposition}

The remainder of this section contains the proofs of 
Propositions~\ref{prop:3.1}--\ref{prop:3.6}.

\begin{proof}[Proof of Proposition~\ref{prop:3.1}]
The Hankel determinants $H_n(B_k)$ of the Bernoulli numbers were first 
determined by Al-Salam and Carlitz \cite{AC}; here we use a standard format as
given in \cite[Sect.~7.1]{DJ2}, namely
\begin{equation}\label{3.7}
H_n(B_k)=(-1)^{n(n+1)/2}
\prod_{\ell=1}^n\left(\frac{\ell^4}{4(2\ell+1)(2\ell-1)}\right)^{n+1-\ell},
\end{equation}
so that
\[
\frac{H_n(B_k)}{H_{n-1}(B_k)}
=(-1)^n \prod_{\ell=1}^n\frac{\ell^4}{4(2\ell+1)(2\ell-1)}
=\frac{(-1)^n n!^6}{(2n)!(2n+1)!}.
\]
Comparing this with \eqref{3.2}, we see that we need to show that for all 
$n\geq 1$ we have
\begin{equation}\label{3.8}
\frac{H_n(b_k)}{H_{n-1}(B_k)}=(-1)^n\frac{2\cdot n!^3}{(2n)!}
\cdot{\mathcal H}_{2,n}^{-} =: r_n. 
\end{equation}
Touchard \cite{To} was the first to show, in a slightly different normalization,
that \eqref{2.4} holds for $c_k=B_k$ with
\[
s_n=\frac{1}{2},\qquad t_n = \frac{-n^4}{4(2n+1)(2n-1)}.
\]
Hence by Lemma~\ref{lem:2.5} we are done if we can verify that the sequence
$(r_n)$ satisfies the recurrence relation 
\begin{equation}\label{3.9}
r_{n+1} = -\frac{1}{2}r_n + \frac{n^4}{4(2n+1)(2n-1)}r_{n-1}.
\end{equation}
By direct computation we find $r_1=-1$ and $r_2=1/2$, which holds for both
sides of \eqref{3.8}. Next, if we substitute \eqref{3.8} into \eqref{3.9} and
multiply both sides by $(2n+1)!/n!^3$, we get 
\begin{equation}\label{3.10}
(n+1)^2{\mathcal H}_{2,n+1}^{-} = (2n+1){\mathcal H}_{2,n}^{-}
+ n^2{\mathcal H}_{2,n-1}^{-}.
\end{equation}
On the other hand, from the definition in \eqref{3.1} we have
\begin{equation}\label{3.11}
{\mathcal H}_{2,n+1}^{-} - {\mathcal H}_{2,n}^{-} = \frac{(-1)^n}{(n+1)^2}.
\end{equation}
Replacing $n$ by $n-1$ and combining the resulting identity with \eqref{3.11},
we get
\[
(n+1)^2\left({\mathcal H}_{2,n+1}^{-} - {\mathcal H}_{2,n}^{-}\right)
= -n^2\left({\mathcal H}_{2,n}^{-} - {\mathcal H}_{2,n-1}^{-}\right),
\]
which is equivalent to \eqref{3.10}. Hence we have verified \eqref{3.9}, and
the proof is complete.
\end{proof}

\begin{proof}[Proof of Proposition~\ref{prop:3.2}]
The outline is the same as that of the previous proof.
The evaluation of $H_n\big(E_k(1)\big)$, and in fact for
$H_n\big(E_k(x)\big)$, is also due to Al-Salam and Carlitz \cite{AC} (who used
what is stated as Lemma~\ref{lem:7.1} below), but again
we give it in the equivalent form
\begin{equation}\label{3.12}
H_n\big(E_k(1)\big)=(-1)^{n(n+1)/2}
\prod_{\ell=1}^n\left(\frac{\ell^2}{4}\right)^{n+1-\ell}
\end{equation}
(see \cite[Sect.~7.1]{DJ2}), so that
\[
\frac{H_n\big(E_k(1)\big)}{H_{n-1}\big(E_k(1)\big)}
=(-1)^n \prod_{\ell=1}^n\frac{\ell^2}{4}
=(-1)^n \frac{n!^2}{2^{2n}}.
\]
Comparing this with \eqref{3.3}, we see that we need to show that for all
$n\geq 1$ we have
\begin{equation}\label{3.13}
\frac{H_n(b_k)}{H_{n-1}\big(E_k(1)\big)}
=-\frac{n!}{2^{n-1}}\cdot{\mathcal H}_{1,n}^{-} =: r_n.
\end{equation}
By using the special case $x=1$ in Theorem~1 of \cite{JS}, we see that 
\eqref{2.4} holds for $c_k=E_k(1)$ with
\[
s_n=-\frac{1}{2},\qquad t_n = -\frac{n^2}{4}.
\]
Hence by Lemma~\ref{lem:2.5} we are done if we can verify that the sequence
$(r_n)$ satisfies 
\begin{equation}\label{3.14}
r_{n+1} = \frac{1}{2}r_n + \frac{1}{4}\,n^2r_{n-1}.
\end{equation}
By direct computation we find that $r_1=-1$ and $r_2=-1/2$ hold for both
sides of \eqref{3.13}. If we substitute \eqref{3.13} into \eqref{3.14} and
multiply both sides by $-n!/2^n$, we get the equivalent form
\[
(n+1){\mathcal H}_{1,n+1}^{-} = {\mathcal H}_{1,n}^{-} 
+ n{\mathcal H}_{1,n-1}^{-}.
\]
This, finally, is easy to verify using the definition in \eqref{3.1}. The proof
is now complete.
\end{proof}

\begin{proof}[Proof of Proposition~\ref{prop:3.3} (sketch)]
Since this proof follows again the same outline as before, we only give the two
main ingredients. First, the Hankel determinant of the sequence
$\big(E_{2k+1}(1)\big)_{k\geq 0}$ can be found in \cite[Eq.~(4.56)]{Mi}, or
equivalently in \cite[Sect.~7.1]{DJ2} as
\begin{equation}\label{3.15}
H_n\big(E_{2k+1}(1)\big)=\frac{1}{2^{n+1}}
\prod_{\ell=1}^n\left(\frac{\ell^2(2\ell-1)(2\ell+1)}{4}\right)^{n+1-\ell}.
\end{equation}
Second, using \cite[Eq.~(5.4)]{DJ1} with $\nu=1$ and $x=1$ we see that 
\eqref{2.4} is satisfied with
\[
s_n=\frac{(2n+1)^2}{2},\qquad t_n=\frac{n^2(2n+1)(2n-1)}{4}.
\]
We leave all further details of the proof to the reader.
\end{proof}

\begin{proof}[Proof of Proposition~\ref{prop:3.4} (sketch)]
Once again, the proof proceeds as before, with the first main ingredient being
\begin{equation}\label{3.16}
H_n\big((2^{2k+2}-1)B_{2k+2}\big)=\frac{1}{2^{n+1}}
\prod_{\ell=1}^n\big(\ell^3(\ell+1)\big)^{n+1-\ell},
\end{equation}
which was obtained as Corollary~5.3 in \cite{DJ2}.

Finding the second main ingredient, namely the pair of coefficient sequences
$(s_n)$ and $(t_n)$, is a bit less straightforward than in the previous proofs.
In Theorem~5.1 of \cite{DJ1} with $\nu=2$, the orthogonal polynomials belonging
to the polynomial sequence $\big(E_{2k+2}(\frac{1+x}{2})\big)_{k\geq 0}$ is
given as 
\[
P_{n+1}(y;x) = \big(y+s_n(x)\big)P_n(y;x) - t_n(x)P_{n-1}(y;x),
\]
with appropriate initial conditions, and with
\begin{equation}\label{3.17}
s_n(x)=(2n+1)(n+1)-\frac{x^2-1}{4},\qquad
t_n(x)=\frac{n^2}{4}\left((2n+1)^2-x^2\right).
\end{equation}
Next we note that by some known properties of Euler polynomials, namely
the identity \eqref{3.4} and the fact that $\frac{d}{dx}E_k(x)=kE_{k-1}(x)$,
we have
\[
(2^{2k+2}-1)B_{2k+2} 
= \left.\frac{d}{dx}E_{2k+2}\big(\tfrac{1+x}{2}\big)\right|_{x=1}.
\]
Since $E_{2k+2}(1)=0$ for all $k\geq 0$, by Lemma~\ref{lem:2.4} the 
sequence $\big((2^{2k+2}-1)B_{2k+2}\big)_{k\geq 0}$ has $P_n(y,1)$ as its
associated orthogonal polynomials. This means that by \eqref{3.17} this sequence
satisfies \eqref{2.4} with
\[
s_n=s_n(1)=(2n+1)(n+1),\qquad t_n=t_n(1)=n^3(n+1).
\]
From here on we proceed as in the previous three proofs; once again we leave
the details to the reader.
\end{proof}

\begin{proof}[Proof of Proposition~\ref{prop:3.5}]

With the aim of applying Lemma~\ref{lem:2.4}, we set $c_k(x):=E_{2k+1}(x)$.
Then with \eqref{1.3b},
\begin{equation}\label{3.18}
c_k(\tfrac{1}{2}) = E_{2k+1}(\tfrac{1}{2}) = 2^{-2k-1}E_{2k+1}=0\qquad
(k\geq 0).
\end{equation}
By Theorem~5.1 of \cite{DJ1} with $\nu=1$ we know that the monic orthogonal
polynomials with respect to $E_{2k+1}(\tfrac{x+1}{2})$ are given by
$q_0(y;x)=1$, $q_1(y;x)=y+s_0(x)$, and 
\begin{equation}\label{3.19}
q_{n+1}(y;x) = \big(y+s_n(x)\big)q_n(y;x) - t_n(x)q_{n-1}(y;x)\qquad(n\geq 1),
\end{equation}
where
\[
s_n(x)=(2n+1)(n+\tfrac{1}{2})-\frac{x^2-1}{4},\qquad
t_n(x)=\frac{n^2}{4}\left(4n^2-x^2\right),
\]
and thus
\begin{equation}\label{3.20}
s_n:=s_n(0)=2n^2+2n+\frac{3}{4},\qquad
t_n:=t_n(0)=n^4.
\end{equation}
By Lemma~\ref{lem:2.4} with \eqref{3.18}, the polynomials \eqref{3.19} with
$x=0$ are therefore the monic orthogonal polynomials also for 
\begin{equation}\label{3.21}
c'_k(\tfrac{1}{2}) = (2k+1)E_{2k}(\tfrac{1}{2}) = 2^{-2k}(2k+1)E_{2k},
\end{equation}
where we have used \eqref{1.3b} again. Now, to apply Lemma~\ref{lem:2.5}, we set
\[
r_n:=\frac{H_n\big(2^{-2k}b_k\big)}{H_{n-1}\big(2^{-2k}(2k+1)E_{2k}\big)}
=\left(\frac{1}{2}\right)^{4n}\frac{H_n\big(b_k\big)}{H_{n-1}\big((2k+1)E_{2k}\big)},
\]
where we have used the first identity in \eqref{4.8} below. Then we have
\begin{equation}\label{3.22}
r_n=\left(\frac{1}{2}\right)^{4n}\frac{H_n\big(b_k\big)}{H_n\big((2k+1)E_{2k}\big)}
\cdot\frac{H_n\big((2k+1)E_{2k}\big)}{H_{n-1}\big((2k+1)E_{2k}\big)}.
\end{equation}
Next, by Corollary~5.2 in \cite{DJ2} we have
\[
H_n\big((2k+1)E_{2k}\big)=2^{2n(n+1)}\prod_{\ell=1}^n\ell!^4,
\]
so that
\[
\frac{H_n\big((2k+1)E_{2k}\big)}{H_{n-1}\big((2k+1)E_{2k}\big)}
= 2^{4n}n!^4.
\]
By combining this and \eqref{3.6a} with \eqref{3.22} we see that we are done if
we can show that
\begin{equation}\label{3.23}
r_n= (-1)^n\left(\tfrac{1}{2}\right)^{4n}h_n.
\end{equation}
We do this by applying Lemma~\ref{lem:2.5} to the sequence \eqref{3.21}, and we
get 
\begin{equation}\label{3.24}
r_{n+1} = -s_nr_n - t_nr_{n-1},
\end{equation}
with the various terms given by \eqref{3.23} and \eqref{3.20}. Multiplying both
sides of \eqref{3.24} by $(-16)^{n+1}$, we see that it is equivalent to 
\eqref{3.6b}. Finally, the initial values for $n=1, 2$ are again easy to
establish by direct computation, which completes the proof.
\end{proof}

\begin{proof}[Proof of Proposition~\ref{prop:3.6}]
Once again we proceed as in the earlier proofs, and note that in 
\cite[Eq.~(H12)]{Ha16} it was shown that
\begin{equation}\label{3.25}
H_n\left(\frac{E_{k+1}(1)}{(k+1)!}\right)=\frac{(-1)^{n(n+1)/2}}{2^{n+1}}
\prod_{\ell=1}^n\left(\frac{1}{4(2\ell-1)(2\ell+1)}\right)^{n+1-\ell},
\end{equation}
written in the format used in \cite[Sect.~7.1]{DJ2}. To simplify notation, we
set $c_k:=E_{k+1}(1)/(k+1)!$; then we get with \eqref{3.25},
\[
\frac{H_n(c_k)}{H_{n-1}(c_k)} 
= \frac{(-1)^n}{2}\prod_{\ell=1}^n\frac{1}{4(2\ell-1)(2\ell+1)} 
= \frac{(-1)^n n!^2}{2(2n)!(2n+1)!},
\]
where the right-most term follows from some straightforward manipulations. 
Comparing this with \eqref{3.6c}, we see that for all $n\geq 1$ we need to show
that
\begin{equation}\label{3.26}
r_n:=\frac{H_n(c_{k-1})}{H_{n-1}(c_k)}
=\frac{H_n(c_k)}{H_{n-1}(c_k)}\cdot\frac{H_n(c_{k-1})}{H_n(c_k)}
=(-1)^n\frac{n!}{(2n)!}.
\end{equation}
The identity \eqref{3.4} implies that $c_k=0$ for all
odd $k\geq 1$. By the theory of classical orthogonal polynomials we then have
$s_n=0$, $n\geq 1$, for the polynomials in \eqref{2.4}; see, e.g.,
Definition~4.1 and Theorem~4.3 in \cite[pp.~20--21]{Chi}. Furthermore, from
\eqref{3.25} we get
\[
t_n = \frac{-1}{4(2n-1)(2n+1)},
\]
and the recurrence relation \eqref{2.8} reduces to $r_{n+1}=-t_nr_{n-1}$. It is
now easy to verify that $r_n$, as given in \eqref{3.26}, satisfies this 
recurrence. Finally, the initial values $r_1=-1/2$ and $r_2=1/12$ can be
verified by direct computation, which completes the proof.
\end{proof}

\section{Further auxiliary results}\label{sec4}

In this section we quote a few known results that will be required in the
proofs of more Hankel determinant identities in later sections.

As we have seen, this paper is mainly concerned with finding
Hankel determinants of right-shifted sequences. 
Interestingly, for the proofs of some more such identities, known
results on {\it left\/}-shifted sequences turn out to be useful; we summarize 
them now.

As before, let ${\bf c}=(c_0,c_1,\ldots)$ be a given sequence, and let
$P_n(y)$, $n=0,1,\ldots$, be the polynomials orthogonal with respect to
${\bf c}$, satisfying the recurrence relation \eqref{2.4}.
Following \cite{MWY}, we consider the infinite band matrix
\begin{equation}\label{4.0}
J:=\begin{pmatrix}
-s_{0} & 1 & 0 & 0 &\cdots \\
t_{1} & -s_{1} & 1 & 0 & \cdots \\
0 & t_{2} & -s_{2} & 1 & \cdots \\
\vdots & \vdots & \ddots & \ddots & \ddots
\end{pmatrix}.
\end{equation}
Furthermore, for each $n\geq 0$ let $J_{n}$ be the $(n+1)$th leading principal
submatrix of $J$ and let
\begin{equation}\label{4.1}
d_{n}:=\det{J_{n}},
\end{equation}
so that $d_0=-s_0$. We also set $d_{-1}=1$ by convention, and furthermore, 
using elementary determinant operations, we get from \eqref{4.0} the recurrence
relation
\begin{equation}\label{4.1a}
d_{n+1} = -s_{n+1}d_{n} - t_{n+1}d_{n-1}.
\end{equation}
We can now quote the following results.

\begin{lemma}[{\cite[Prop.~1.2]{MWY}}]\label{lem:4.1}
With notation as above, for a given sequence ${\bf c}$ we have
\begin{equation}\label{4.2}
H_{n}(c_{k+1})=H_{n}(c_{k})\cdot d_{n},
\end{equation}
and
\begin{equation}\label{4.3}
H_{n}(c_{k+2})=H_{n}(c_{k})\cdot D_n,\qquad\hbox{where}\quad
D_n:=\left(\prod_{\ell=1}^{n+1}t_{\ell}\right)
\cdot\sum_{\ell=-1}^{n}\frac{d_{\ell}^{2}}{\prod_{j=1}^{\ell+1}t_{j}}.
\end{equation}
\end{lemma}

\begin{lemma}[{\cite[Eq.~(2.4)]{Ha16}}]\label{lem:4.2}
For a given sequence ${\bf c}$ and $(s_n)$ as defined above, we have
\begin{equation}\label{4.4}
s_{n}=-\frac{1}{H_{n-1}(c_{k+1})}
\left(\frac{H_{n-1}(c_{k})H_{n}(c_{k+1})}{H_{n}(c_{k})}
+\frac{H_{n}(c_{k})H_{n-2}(c_{k+1})}{H_{n-1}(c_{k})}\right).
\end{equation}
\end{lemma}

The next lemma is about determinants of ``checkerboard matrices", namely 
matrices in which every other entry vanishes. This result can be found in 
\cite{CK} as Lemmas~5 and~6, and covers more general matrices than just Hankel
matrices.

\begin{lemma}[{\cite[Lemmas~5, 6]{CK}}]\label{lem:4.3}
Let $M=\left(M_{i,j}\right)_{0\leq i,j\leq n-1}$ be a matrix.
If $M_{i,j}=0$ whenever $i+j$ is odd, then
\begin{equation}\label{4.5}
\det_{0\leq i.j\leq n-1}(M_{i,j})
=\det_{0\leq i.j\leq\lfloor(n-1)/2\rfloor}(M_{2i,2j})
\cdot\det_{0\leq i.j\leq\lfloor(n-2)/2\rfloor}(M_{2i+1,2j+1}).
\end{equation}
If $M_{i,j}=0$ whenever $i+j$ is even, then for even $n$ we have 
\begin{equation}\label{4.6}
\det_{0\leq i.j\leq n-1}(M_{i,j})=
(-1)^{n/2}\det_{0\leq i.j\leq\lfloor(n-1)/2\rfloor}(M_{2i+1,2j})
\cdot\det_{0\leq i.j\leq\lfloor(n-2)/2\rfloor}(M_{2i,2j+1}), 
\end{equation}
while for odd $n$ we have
\begin{equation}\label{4.7}
\det_{0\leq i.j\leq n-1}(M_{i,j}) = 0.
\end{equation}
\end{lemma}

Lemma~\ref{lem:4.3} is best explained by way of an example.

\medskip
\noindent
{\bf Example~1.} By \eqref{4.5} we have
\[
\det\begin{pmatrix}a & 0 & b & 0 & c\\
0 & {\bf d} & 0 & {\bf e} & 0\\
f & 0 & g & 0 & h\\
0 & {\bf i} & 0 & {\bf j} & 0\\
k & 0 & l & 0 & m
\end{pmatrix}=\det\begin{pmatrix}a & b & c\\
f & g & h\\
k & l & m
\end{pmatrix}\cdot\det\begin{pmatrix}{\bf d} & {\bf e}\\
{\bf i} & {\bf j}
\end{pmatrix},
\]
\[
\det\begin{pmatrix}a & 0 & b & 0\\
0 & {\bf d} & 0 & {\bf e}\\
f & 0 & g & 0\\
0 & {\bf i} & 0 & {\bf j}
\end{pmatrix}=\det\begin{pmatrix}a & b\\
f & g
\end{pmatrix}\cdot\det\begin{pmatrix}{\bf d} & {\bf e}\\
{\bf i} & {\bf j}
\end{pmatrix}.
\]

\medskip
We conclude this section with another useful property of Hankel determinants. 
It follows from basic determinant operations involving the matrix \eqref{1.0};
details can be found in \cite{DJ1}.

\begin{lemma}\label{lem:4.4}
Let $x$ be a variable or a complex number and $(c_0, c_1,\ldots)$ a sequence.
Then
\begin{equation}\label{4.8}
H_n(x^kc_k) = x^{n(n+1)}H_n(c_k)\quad\hbox{and}\quad 
H_n(x\cdot c_k) = x^{n+1}H_n(c_k).
\end{equation}
\end{lemma}

\section{Hankel determinant identities, II}\label{sec5}

This section will be concerned with further Hankel determinant identities
involving subsequences of $E_k(1)$. See also Table~\ref{tab:1} for the first few values. In addition to the identity \eqref{3.15} 
for $H_{n}\left(E_{2k+1}(1)\right)$, we have
\begin{equation}\label{5.2}
H_{n}\left(E_{2k+3}(1)\right) = \left(\frac{-1}{4}\right)^{n+1}
\prod_{\ell=1}^{n}\left(\frac{\ell(\ell+1)(2\ell+1)^2}{4}\right)^{n+1-\ell},
\end{equation}
which was obtained in \cite[Eq.~(4.57)]{Mi}; see also \cite[Sect.~7.1]{DJ2}. 
To state the results in this section, we require another harmonic-type 
sequence, namely
\begin{equation}\label{5.3}
\overline{{\mathcal H}}_n := \sum_{j=0}^n\frac{1}{2j+1}
= {\mathcal H}_{2n+2}-\frac{1}{2}{\mathcal H}_{n+1}.
\end{equation}

\begin{proposition}\label{prop:5.1}
For $n\geq 0$ we have
\begin{equation}\label{5.4}
H_{n}\big(E_{2k+5}(1)\big)=\left(\frac{1}{2}\right)^{n+1}
\overline{{\mathcal H}}_{n+1}\cdot
\prod_{\ell=1}^{n+1}\left(\frac{\ell^2(2\ell+1)(2\ell-1)}{4}\right)^{n+2-\ell}.
\end{equation}
\end{proposition}

\begin{proof}
Our main tool will be Lemma~\ref{lem:4.1}. Using \eqref{3.15} and \eqref{5.2}
and some straightforward but tedious manipulations, we obtain
\begin{equation}\label{5.5}
d_n := \frac{H_{n}\big(E_{2k+3}(1)\big)}{H_{n}\big(E_{2k+1}(1)\big)}
= \left(-\frac{1}{4}\right)^{n+1}(2n+2)!.
\end{equation}
Next, we note that the orthogonal polynomial \eqref{2.4} with respect to
$c_k:=E_{2k+1}(1)$ has
\begin{equation}\label{5.6}
t_n = \frac{n^2(2n+1)(2n-1)}{4},
\end{equation} 
which follows from \eqref{3.15}. Now \eqref{4.3} with the convention $d_{-1}=1$
gives
\begin{equation}\label{5.7}
H_{n}\big(E_{2k+5}(1)\big) = H_{n}\big(E_{2k+1}(1)\big)\cdot
\left(\prod_{\ell=1}^{n+1}t_{\ell}\right)\cdot\left(1+\sum_{\ell=0}^n
\frac{d_{\ell}^2}{\prod_{j=1}^{\ell+1}t_{j}}\right).
\end{equation}
Using \eqref{5.5} and \eqref{5.6}, then after some straightforward 
manipulations the right-most term in large parentheses in \eqref{5.7} turns 
out to be
\[
1+\sum_{\ell=0}^n\frac{1}{2\ell+3}=\sum_{j=0}^{n+1}\frac{1}{2j+1} 
= \overline{{\mathcal H}}_{n+1}.
\]
Substituting this and \eqref{3.15} and \eqref{5.6} into \eqref{5.7}, we easily
obtain the desired identity \eqref{5.4}.
\end{proof} 

Proposition~\ref{prop:5.1} will now be used in the proof of the next result.

\begin{proposition}\label{prop:5.2}
For $n\geq 1$ we have
\begin{equation}\label{5.8}
\frac{H_{2n}\big(E_{k+3}(1)\big)}{H_{2n-1}\big(E_{k+3}(1)\big)}
= -\frac{(n+1)(2n+1)!^2}{2^{4n+2}}
\end{equation}
and
\begin{align}
&H_{2n}\big(E_{k+3}(1)\big) 
=\frac{(-1)^{n+1}\cdot\overline{{\mathcal H}}_n}{2^{3n+2}}\cdot\prod_{\ell=1}^n
\left(\frac{\ell^3(\ell+1)(2\ell+1)^3(2\ell-1)}{16}\right)^{n+1-\ell},\label{5.9}\\
&H_{2n-1}\big(E_{k+3}(1)\big) 
=\frac{(-1)^n\cdot 2^n\cdot\overline{{\mathcal H}}_n}{(n+1)(2n+1)!^2}\label{5.10}\\
&\qquad\qquad\qquad\qquad\cdot\prod_{\ell=1}^n\left(\frac{\ell^3(\ell+1)(2\ell+1)^3(2\ell-1)}{16}\right)^{n+1-\ell}.\nonumber
\end{align}
\end{proposition}

\begin{proof}
By \eqref{3.4} and the fact that $B_{2j+1}=0$ for $j\geq 1$, we have 
$E_{2j}(1)=0$ for any $j\geq 1$. This means that the determinant 
$H_n\big(E_{k+3}(1)\big)$ is of ``checkerboard type", and we can apply the first
case of Lemma~\ref{lem:4.3}. Using \eqref{4.5} with $n$ replaced by $2n+1$ and 
by $2n$, respectively, and keeping in mind that $i+j=k$, we get the two 
identities
\begin{align}
H_{2n}\big(E_{k+3}(1)\big) 
&= H_{n}\big(E_{2k+3}(1)\big)\cdot H_{n-1}\big(E_{2k+5}(1)\big),\label{5.11}\\ 
H_{2n-1}\big(E_{k+3}(1)\big)
&= H_{n-1}\big(E_{2k+3}(1)\big)\cdot H_{n-1}\big(E_{2k+5}(1)\big).\label{5.12}
\end{align}
Taking the quotient of these identities and using \eqref{5.2}, we obtain
\begin{align*}
\frac{H_{2n}\big(E_{k+3}(1)\big)}{H_{2n-1}\big(E_{k+3}(1)\big)}
&=\frac{H_{n}\big(E_{2k+3}(1)\big)}{H_{n-1}\big(E_{2k+3}(1)\big)}\\
&=-\frac{n(n+1)(2n+1)^2}{16}
\prod_{\ell=1}^{n-1}\frac{\ell(\ell+1)(2\ell+1)^2}{4},
\end{align*}
which gives \eqref{5.8} after some easy manipulations. The identities 
\eqref{5.9} and \eqref{5.10} follow immediately from \eqref{5.11} and
\eqref{5.12}, respectively, upon using \eqref{5.2} and \eqref{5.4}.
Alternatively, \eqref{5.10} can be obtained by combining \eqref{5.8} and
\eqref{5.9}.
\end{proof}

\section{Hankel determinant identities, III}\label{sec6}

In this section we are going to prove three more identities that are similar in
nature to the results in Section~\ref{sec3}. However, while in the proofs of those
results we were able to use known orthogonal polynomials belonging to the
relevant sequences $(c_0, c_1, \ldots)$, in this section we still need to
determine the coefficients $s_n$ occurring in \eqref{2.4}.

\begin{proposition}\label{prop:6.1}
For all $n\geq 0$ we have
\begin{equation}\label{6a.1}
H_n(B_{2k})=(-1)^n\frac{(4n+3)!}{(n+1)\cdot(2n+1)!^3}\cdot{\mathcal H}_{2n+1}
\cdot H_n(B_{2k+2}).
\end{equation}
\end{proposition}

This is the identity \eqref{1.8} in the Introduction. 

\begin{proposition}\label{prop:6.2}
For all $n\geq 0$ we have
\begin{equation}\label{6.1}
H_n\big((2k+1)B_{2k}\big) = \frac{(-1)^n(2n+2)!}{n!(n+1)!^3}\cdot
\left({\mathcal H}_n+{\mathcal H}_{n+1}\right)\cdot H_n\big((2k+3)B_{2k+2}\big).
\end{equation}
\end{proposition}

The next identity does not contain harmonic or generalized harmonic
numbers; but still, it belongs to the same category as the previous two 
identities.

\begin{proposition}\label{prop:6.3}
If the sequence $(b_0, b_1, \ldots)$ is defined by
\[
b_k:=\begin{cases}
0,& k=0,\\
E_{2k-1}(1)/(2k-1)!,& k\geq 1,
\end{cases}
\]
then for all $n\geq 1$ we have
\begin{equation}\label{6b.1}
H_n(b_k) = (-1)^n\frac{(4n+2)!}{(2n-1)!}
\cdot H_n\left(\frac{E_{2k+1}(1)}{(2k+1)!}\right).
\end{equation}
\end{proposition}

The Hankel determinants on the right of \eqref{6a.1}, \eqref{6.1}, and 
\eqref{6b.1} are given explicitly by \eqref{1.6}, \eqref{6.3}, and
\eqref{6a.7}, respectively.

We prove these three results in sequence. First, for the proof of 
Proposition~\ref{prop:6.1} we require the following lemma.

\begin{lemma}\label{lem:6.4}
If $P_n(y)$, $n=0, 1, \ldots$, are the monic orthogonal polynomials with respect
to the sequence $\big(B_{2k+2}\big)_{k\geq 0}$, then
\begin{equation}\label{6a.2}
s_n=\frac{(n+1)(2n+1)(4n^2+6n+1)}{(4n+1)(4n+5)},\qquad
t_n=\frac{n^3(n+1)(2n-1)(2n+1)^3}{(4n-1)(4n+1)^2(4n+3)},
\end{equation}
with $s_n, t_n$ as in \eqref{2.4}.
\end{lemma}

\begin{proof}
In view of Corollary~\ref{cor:2.3}, the identities \eqref{1.5} and \eqref{1.6}
immediately give $t_n$ in \eqref{6.2}. Next, by \eqref{4.4} and using
\eqref{1.5}--\eqref{1.7}, we get after some easy manipulations,
\[
d_n=\frac{H_n(B_{2k+4})}{H_n(B_{2k+2})} = \left(\frac{-1}{5}\right)^{n+1}
\prod_{\ell=1}^{n}\left(\frac{(\ell+1)^2(2\ell+3)(4\ell-1)(4\ell+1)}
{\ell^2(2\ell-1)(4\ell+3)(4\ell+5)}\right)^{n+1-\ell},
\]
so that 
\[
\frac{d_n}{d_{n-1}} = \frac{-1}{5}\prod_{\ell=1}^{n}
\frac{(\ell+1)^2(2\ell+3)(4\ell-1)(4\ell+1)}{\ell^2(2\ell-1)(4\ell+3)(4\ell+5)}
= -\frac{(n+1)^2(2n+1)(2n+3)}{(4n+3)(4n+5)},
\]
where the right-most term also follows by easy manipulations.
Using this and the identity \eqref{4.1a}, along with the second equation in
\eqref{6a.2}, we get
\begin{align*}
s_n &= -\frac{d_n}{d_{n-1}}-t_n\frac{d_{n-2}}{d_{n-1}}\\
& =\frac{(n+1)^{2}(2n+1)(2n+3)}{(4n+3)(4n+5)}+\frac{n^{3}(n+1)(2n-1)(2n+1)^{3}}{(4n-1)(4n+1)^{2}(4n+3)}\cdot\frac{(4n-1)(4n+1)}{n^{2}(2n-1)(2n+1)}\\
& =\frac{(n+1)(2n+1)}{(4n+1)(4n+5)}(4n^{2}+6n+1),
\end{align*}
where the final term is again the result of some easy manipulations. This 
completes the proof of Lemma~\ref{lem:6.4}.
\end{proof}

\begin{proof}[Proof of Proposition~\ref{prop:6.1}]
We proceed as in the proofs in Section~3. By \eqref{1.5} and \eqref{1.6} we
have
\begin{align*}
\frac{H_n(B_{2k+2})}{H_{n-1}(B_{2k+2})} &= \frac{1}{6}\prod_{\ell=1}^n
\frac{\ell^3(\ell+1)(2\ell-1)(2\ell+1)^3}{(4\ell-1)(4\ell+1)^2(4\ell+3)}\\
&= \frac{(2n+1)^4(n+1)(2n)!^6}{(4n+1)!(4n+3)!},
\end{align*}
where the last term follows after some tedious but straightforward
manipulations. Therefore, in order to prove \eqref{6a.1}, we need to show that
\begin{align}
r_n &:= \frac{H_n(B_{2k})}{H_{n-1}(B_{2k+2})}
=\frac{H_n(B_{2k})}{H_n(B_{2k+2})}\cdot\frac{H_n(B_{2k+2})}{H_{n-1}(B_{2k+2})}\label{6a.3} \\
&= (-1)^n\frac{(2n)!^2(2n+1)!}{(4n+1)!}{\mathcal H}_{2n+1}.\nonumber
\end{align}
Direct computation shows that this holds for $n=1$ and $n=2$, and by 
Lemma~\ref{lem:2.5} we are done if the right-most term in \eqref{6a.3} satisfies
\begin{equation}\label{6a.4}
r_{n+1} = -s_nr_n-t_nr_{n-1},
\end{equation}
with $s_n, t_n$ as in \eqref{6a.2}. For greater ease of notation we now set
$h_n:={\mathcal H}_{2n+1}$. Clearing the denominators and all common factors,
we see that \eqref{6a.4} is equivalent to 
\begin{equation}\label{6a.5}
(n+1)(2n+3)(4n+1)h_{n+1} = (4n+3)(4n^2+6n+1)h_n-n(2n+1)(4n+5)h_{n-1}.
\end{equation}
On the other hand, by the definition \eqref{1.9} of the harmonic numbers we
have
\[
h_{n+1}-h_n=\frac{4n+5}{(2n+2)(2n+3)},\qquad
h_n-h_{n-1}=\frac{4n+1}{2n(2n+1)},
\]
so that
\[
\frac{(2n+2)(2n+3)}{4n+5}\left(h_{n+1}-h_n\right)
= \frac{2n(2n+1)}{4n+1}\left(h_n-h_{n-1}\right).
\]
If we collect the coefficients of $h_n$ and simplify, we see that this last
identity is equivalent to \eqref{6a.5}; this completes the proof.
\end{proof}

Next, for the proof of Proposition~\ref{prop:6.2} we require the following
lemma; its proof is more involved than that of Lemma~\ref{lem:6.4}.

\begin{lemma}\label{lem:6.2} 
If $P_n(y)$, $n=0, 1, \ldots$, are the monic orthogonal polynomials with 
respect to the sequence $\big((2k+3)B_{2k+2}\big)_{k\geq 0}$, then
\begin{equation}\label{6.2}
s_n=\frac{(n+1)^{2}(2n^2+4n+1)}{(2n+1)(2n+3)},\qquad
t_n=\frac{n^{3}(n+1)^{3}}{4(2n+1)^2},
\end{equation}
with $s_n, t_n$ as in \eqref{2.4}.
\end{lemma}

\begin{proof}
In \cite[Cor.~6.3]{DJ2} we showed that
\begin{equation}\label{6.3}
H_n\big((2k+3)B_{2k+2}\big) = \frac{1}{2^{n+1}}\prod_{\ell=1}^{n}
\left(\frac{\ell^{3}(\ell+1)^3}{4(2\ell+1)^2}\right)^{n+1-\ell},
\end{equation}
which immediately gives $t_n$ in \eqref{6.2}.

To obtain the formula for $s_n$ in \eqref{6.2}, we consider the sequence
\begin{equation}\label{6.4}
c_k = c_k(x) := B_{2k+1}(\tfrac{1+x}{2}),\qquad k=0, 1,\ldots,
\end{equation}
and note that
\begin{equation}\label{6.5}
(2k+3)B_{2k+2}=2c_{k+1}'(-1)\quad\hbox{and}\quad
c_{k+1}(-1) = B_{2k+3}=0,\quad k=0, 1,\ldots,
\end{equation}
where $c_{k+1}'(x)$ denotes the derivative. The identities in \eqref{6.5}
will allow us to use Lemma~\ref{lem:2.4}, with $k+1$ in place of $k$.

In \cite[Theorem~4.1]{DJ1} it was shown that the orthogonal polynomials
with respect to the sequence $\big(c_k(x)\big)$ in \eqref{6.4} are given by
\[
W_{n+1}(y;x)=\big(y+\sigma_n(x)\big)\cdot W_{n}(y;x)-\tau_n(x)W_{n-1}(y;x),
\]
where 
\begin{equation}\label{6.6}
\sigma_n(x)=\binom{n+1}{2}-\frac{x^{2}-1}{4},\qquad 
\tau_n(x)=\frac{n^{4}(n^{2}-x^{2})}{4(2n+1)(2n-1)}. 
\end{equation}
We are now going to use Lemmas~\ref{lem:4.1} and~\ref{lem:4.2}, with
$\sigma_n(x)$, $\tau_n(x)$ in place of $s_n$ and $t_n$, respectively, and
note that $d_n$ in \eqref{4.1} is a function of $x$; we write it as 
$d_n(x)$. We first observe that with \eqref{4.1} and \eqref{6.6} we have
$d_0(x)=-\sigma_0(x)=(1-x^2)/4$, so that $d_0(-1)=0$. Similarly,
\begin{equation}\label{6.7}
d_1(x)=\det\begin{pmatrix}
-\sigma_{0} & 1\\
\tau_{1} & -\sigma_{1}
\end{pmatrix}
=\det\begin{pmatrix}
\frac{x^{2}-1}{4} & 1\\
\frac{1-x^{2}}{12} & \frac{x^{2}-5}{4}
\end{pmatrix}
=(x^{2}-1)\cdot\frac{3x^{2}-11}{48},
\end{equation}
so that $d_1(-1)=0$. Further, for $n\geq 2$ we use Lemma~6.2 in \cite{DJ2},
where it was shown that 
\begin{equation}\label{6.8}
\lim_{x\rightarrow-1}\frac{d_{n}(x)}{x^{2}-1}
=\frac{1}{4}\left(-\frac{2}{3}\right)^{n}\prod_{\ell=2}^n
\left(\frac{(\ell+1)^2(2\ell-1)}{\ell(\ell-1)(2\ell+1)}\right)^{n+1-\ell},
\end{equation}
so altogether we have
\begin{equation}\label{6.9}
d_n(-1) = 0\quad\hbox{for all}\quad n\geq 0,
\end{equation}
while $d_{-1}(-1)=1$ by convention. With \eqref{6.8} we also obtain
\begin{equation}\label{6.10}
\lim_{x\rightarrow-1}\frac{d_{n-1}}{d_{n}} 
=\lim_{x\rightarrow-1}\frac{d_{n-1}/(x^{2}-1)}{d_{n}/(x^{2}-1)}
= -\frac{2(2n+1)}{n(n+1)^{2}},
\end{equation}
where the last equation is obtained after some easy manipulations. With
\eqref{6.7} and the identity for $d_0(x)$ we see that \eqref{6.10} holds
for all $n\geq 1$. 

We also require the first identity in \eqref{4.3}, with
\[
D_n=D_n(x):=\left(\prod_{\ell=1}^{n+1}\tau_{\ell}(x)\right)
\cdot\sum_{\ell=-1}^{n}\frac{d_{\ell}(x)^2}{\prod_{j=1}^{\ell+1}\tau_j(x)}.
\]
Then by \eqref{6.9} and the second identity in \eqref{6.6} we have
\begin{equation}\label{6.11}
\lim_{x\rightarrow-1}\frac{D_n(x)}{D_{n-1}(x)}=\tau_{n+1}(-1)
=\frac{n(n+1)^{4}(n+2)}{4(2n+1)(2n+3)}.
\end{equation}

For the final stage of the proof, we let the orthogonal polynomials with 
respect to $\big(c_{k+1}(x)\big)_{k\geq 0}$ be given by 
\[
Q_{n+1}(y;x)=\big(y+S_n(x)\big)\cdot Q_n(y;x)-T_n(x)\cdot Q_{n-1}(y;x).
\]
Then by Lemmas~\ref{lem:4.1} and~\ref{lem:4.2}, with $s_n$ replaced by $S_n(x)$
and $c_k$ by $c_{k+1}(x)$, we get the following expression for $S_n(x)$; for
greater ease of notation we suppress the variable $x$.
\begin{align*}
S_n &=\frac{-1}{H_{n-1}(c_{k+2})}
\left(\frac{H_{n-1}(c_{k+1})H_{n}(c_{k+2})}{H_{n}(c_{k+1})}
+\frac{H_{n}(c_{k+1})H_{n-2}(c_{k+2})}{H_{n-1}(c_{k+1})}\right)\\
& =\frac{-1}{H_{n-1}(c_{k})D_{n-1}}
\left(\frac{H_{n-1}(c_{k})d_{n-1}H_{n}(c_{k})D_{n}}{H_{n}(c_{k})d_{n}}
+\frac{H_{n}(c_{k})d_{n}H_{n-2}(c_{k})D_{n-2}}{H_{n-1}(c_{k})d_{n-1}}\right).
\end{align*}
By \eqref{2.6} we have
\[
\frac{H_{n}(c_{k})d_{n}H_{n-2}(c_{k})}{H_{n-1}(c_{k})^2}=\tau_n,
\]
and thus
\begin{equation}\label{6.12}
S_n(x)=-\frac{D_n(x)}{D_{n-1}(x)}\cdot\frac{d_{n-1}(x)}{d_{n}(x)}
-\tau_n(x)\cdot\frac{D_{n-2}(x)}{D_{n-1}(x)}\cdot\frac{d_{n}(x)}{d_{n-1}(x)}.
\end{equation}
Finally, using Lemma~\ref{2.4} with \eqref{6.5}, and then \eqref{6.6}, 
\eqref{6.10} and \eqref{6.11} substituted into \eqref{6.12}, we get after some
easy manipulations,
\[
s_n = \lim_{x\rightarrow-1}S_n(x)
=\frac{(n+1)^{2}(n+2)}{2(2n+3)}+\frac{n(n+1)^2}{2(2n+1)},
\]
which immediately gives the desired first identity in \eqref{6.2}.
\end{proof}

\begin{proof}[Proof of Proposition~\ref{prop:6.2}]
We proceed as in the proofs in Section~\ref{sec3}. From \eqref{6.3} we obtain
\[
\frac{H_n\big((2k+3)B_{2k+2}\big)}{H_{n-1}\big((2k+3)B_{2k+2}\big)}
= \frac{1}{2}\prod_{\ell=1}^{n}\frac{\ell^{3}(\ell+1)^3}{4(2\ell+1)^2}
= \frac{n!^5(n+1)!^3}{2(2n+1)!^2}.
\]
Therefore, in order to prove \eqref{6.1}, we need to show that
\begin{equation}\label{6.13}
\frac{H_n\big((2k+1)B_{2k}\big)}{H_{n-1}\big((2k+3)B_{2k+2}\big)}
=(-1)^n\frac{n!^3(n+1)!}{(2n+1)!}\left({\mathcal H}_n+{\mathcal H}_{n+1}\right).
\end{equation}
The cases $n=1$ and $n=2$ can be verified by direct computations. By 
Lemma~\ref{lem:2.5} with $\alpha=1$, the left-hand side satisfies the identity
\eqref{2.8}, with $s_n, t_n$ given by \eqref{6.2}. Therefore we are done if we
can show that the right-hand side of \eqref{6.13} satisfies the same recurrence
relation. That is, with $h_n:={\mathcal H}_n+{\mathcal H}_{n+1}$ we need to
show
\begin{equation}\label{6.14}
\frac{(n+2)!(n+1)!^3}{(2n+3)!}h_{n+1}
=\frac{(n+1)^3(2n^2+4n+1)n!^4}{(2n+1)(2n+3)(2n+1)!}h_n
-\frac{(n+1)^3n!^4}{4(2n+1)^2(2n-1)!}h_{n-1}.
\end{equation}
Now, by definition of the harmonic numbers we have
\[
h_{n+1}-h_n = \frac{1}{n+1}+\frac{1}{n+2} = \frac{2n+3}{(n+1)(n+2)},
\]
and thus
\[
\frac{(n+1)(n+2)}{2n+3}\left(h_{n+1}-h_n\right)
=\frac{n(n+1)}{2n+1}\left(h_n-h_{n-1}\right),
\]
or equivalently,
\[
\frac{(n+1)(n+2)}{2n+3}h_{n+1}
=\frac{2(n+1)(2n^2+4n+1)}{(2n+1)(2n+3)}h_n - \frac{n(n+1)}{2n+1}h_{n-1}.
\]
Finally, upon multiplying both sides of this last identity by
$n!(n+1)!^3/(2n+2)!$, we see that it is equivalent to \eqref{6.14}.
This completes the proof.
\end{proof}

To conclude this section, we prove Proposition~\ref{prop:6.3}, beginning with
the following lemma.

\begin{lemma}\label{lem:6.6}
If $P_n(y)$, $n=0, 1, \ldots$, are the monic orthogonal polynomials with
respect to the sequence $\big(E_{2k+1}(1)/(2k+1)!\big)_{k\geq 0}$, then for $n\geq 1$,
\begin{equation}\label{6a.6}
s_n=\frac{1}{2(4n-1)(4n+3)},\quad
t_n=\frac{1}{16(4n-3)(4n-1)^2(4n+1)},
\end{equation}
with $s_n, t_n$ as in \eqref{2.4}.
\end{lemma}

\begin{proof}
The proof follows the same outline as that of Lemma~\ref{lem:6.4}. Let
$c_k:=E_{2k+1}(1)/(2k+1)!$; then by the identities (H13) and (H22),
respectively, in \cite{Ha16} (see also \cite[Sect.~7.1]{DJ2}) we have
\begin{align}
H_n(c_k) &= \left(\frac{1}{2}\right)^{n+1}\prod_{\ell=1}^n
\left(\frac{1}{16(4\ell-3)(4\ell-1)^2(4\ell+1)}\right)^{n+1-\ell},\label{6a.7}\\
H_n(c_{k+1}) &= \left(\frac{-1}{24}\right)^{n+1}\prod_{\ell=1}^n
\left(\frac{1}{16(4\ell-1)(4\ell+1)^2(4\ell+3)}\right)^{n+1-\ell},\nonumber
\end{align}
and thus
\begin{align*}
d_n&:=\frac{H_n(c_{k+1})}{H_n(c_k)}
=\left(\frac{-1}{12}\right)^{n+1}\prod_{\ell=1}^n
\left(\frac{(4\ell-3)(4\ell-1)}{(4\ell+1)(4\ell+3)}\right)^{n+1-\ell}\\
&=\left(\frac{-1}{12}\right)^{n+1}3^n
\prod_{\ell=1}^{n}\frac{1}{(4\ell+1)(4\ell+3)},
\end{align*}
where the second line is a result of cancellations in the product in the first
line. Upon further simplification, we get
\[
d_n = (-1)^{n+1}\frac{(2n+2)!}{(4n+4)!},
\]
and therefore
\begin{equation}\label{6a.8}
\frac{d_n}{d_{n-1}} = -\frac{(2n+2)!}{(4n+4)!}\cdot\frac{(4n)!}{(2n)!}
= \frac{-1}{4(4n+1)(4n+3)}.
\end{equation}
Next, from \eqref{6a.7} we get the second identity in \eqref{6a.6}, and with 
this and \eqref{4.1a} we obtain for $n\geq 1$,
\[
s_n=-\frac{d_n}{d_{n-1}}-t_n\frac{d_{n-2}}{d_{n-1}}
= \frac{1}{2(4n-1)(4n+3)},
\]
where we have used \eqref{6a.8} and then simplified. This is the first 
identity in \eqref{6a.6}, while the second one follows from \eqref{6a.7}, by
\eqref{2.6}. 
\end{proof}
 
Although this will not be needed here, we note that the identity for $s_n$ in 
\eqref{6a.6} does not hold for $n=0$. In fact, we can easily compute (or check
the Table~\ref{tab:1}) that
$c_0=1/2$ and $c_1=-1/24$. Then, by \eqref{2.3},
\[
P_1(y)=\frac{1}{1/2}\det\begin{pmatrix}
\frac{1}{2} & -\frac{1}{24} \\
1 & y \\
\end{pmatrix}
= y + \frac{1}{12},
\]
so that $s_0=1/12$.

\begin{proof}[Proof of Proposition~\ref{prop:6.3}]
Once again we proceed as in the proofs of most of the previous propositions.
With $c_k$ as in the proof of Lemma~\ref{lem:6.6}, from \eqref{6a.7} we get
\[
\frac{H_n(c_k)}{H_{n-1}(c_k)} 
= \frac{1}{2}\prod_{\ell=1}^n\frac{1}{16(4\ell-3)(4\ell-1)^2(4\ell+1)}
= \frac{(2n)!^2}{2(4n+1)(4n)!^2}.
\]
In order to prove \eqref{6b.1}, we therefore need to show that
\begin{equation}\label{6a.9}
r_n:=\frac{H_n(b_k)}{H_{n-1}(c_k)}
= \frac{H_n(b_k)}{H_{n}(c_k)}\cdot\frac{H_n(c_k)}{H_{n-1}(c_k)}
= (-1)^n\frac{(2n+1)!}{2(4n-1)!}.
\end{equation}
Direct computation shows that this holds for $n=1$ and $n=2$, and again by
Lemma~\ref{2.5} we are done if we can show that the right-most term in 
\eqref{6a.9} satisfies the recurrence relation $r_{n+1}=-s_nr_n-t_nr_{n-1}$,
with $s_n$, $t_n$ as in \eqref{6a.6}. It is easy to verify that this is indeed
the case, which completes the proof.
\end{proof}

\section{Hankel determinants of Euler polynomials}\label{sec7}

So far in this paper we have only dealt with Hankel determinants of sequences
of numbers. However, Hankel determinants of {\it polynomial} sequences have
also been studied, going as far back as Al-Salam and Carlitz \cite{AC}. In this
connection the following general result must also be mentioned.

\begin{lemma}\label{lem:7.1}
Let $(c_0, c_1,\ldots)$ be a sequence and $x$ a number or a variable. If
\[
c_k(x) = \sum_{j=0}^k\binom{k}{j}c_jx^{k-j},
\]
then for all $n\geq 0$ we have
\[
H_n(c_k(x)) = H_n(c_k).
\]
\end{lemma}

This can be found, with proof, in \cite{Ju}; it is also mentioned and used in 
various other publications, for instance in \cite[Lemma~15]{Kr1}.
As an immediate consequence of Lemma~\ref{lem:7.1}, together with \eqref{1.3a},
we get the well-known identities
\[
H_n\big(B_k(x)\big)=H_n\big(B_k\big),\qquad
H_n\big(E_k(x)\big)=2^{-n(n+1)}H_n\big(E_k\big)
\]
(see also \cite{AC}), where we used Lemma~\ref{lem:4.4} for the second identity.
In contrast to these identities, in \cite{DJ1} we obtained a number of Hankel
determinant identities for related polynomial sequences, where the determinants
turned out to be functions of $x$, rather than constants. 

In this section we are going to derive a similar identity, but in keeping with
the topic of this paper, it will be for the shifted analogue of a known
evaluation.

\begin{proposition}\label{prop:7.2}
If the polynomial sequence $(b_0(x), b_1(x), \ldots)$ is defined by
\begin{equation}\label{7.0}
b_k(x):=\begin{cases}
0,& k=0,\\
E_{2k-2}\big(\tfrac{x+1}{2}\big),& k\geq 1,
\end{cases}
\end{equation}
then for all $n\geq 1$ we have
\begin{equation}\label{7.1}
H_n\big(b_k(x)\big) = (-1)^{n-1}\frac{4}{n!^2(x^2-1)}\cdot{\mathcal K}_n(x)
\cdot H_n\big(E_{2k}(\tfrac{x+1}{2})\big),
\end{equation}
where ${\mathcal K}_1(x)=1$ and for $n\geq 2$,
\begin{equation}\label{7.2}
{\mathcal K}_n(x) := \sum_{j=1}^{n-1}
\frac{2^2}{3^2-x^2}\cdot\frac{4^2}{5^2-x^2}\cdots\frac{(2j)^2}{(2j+1)^2-x^2}.
\end{equation}
\end{proposition}

As a consequence of Proposition~\ref{prop:7.2} we get the following result.

\begin{proposition}\label{prop:7.3}
If the sequence $(b_0, b_1, \ldots)$ is defined by
\begin{equation}\label{7.2a}
b_k:=\begin{cases}
0,& k=0,\\
E_{2k-2},& k\geq 1,
\end{cases}
\end{equation}
then for all $n\geq 1$ we have
\begin{equation}\label{7.3}
H_n\big(b_k\big) = \frac{(-1)^n}{4^n n!^2}\cdot{\mathcal K}_n
\cdot H_n\big(E_{2k}\big),
\end{equation}
where 
\begin{equation}\label{7.4}
{\mathcal K}_n := \sum_{j=0}^{n-1}\frac{16^j}{(2j+1)^2\binom{2j}{j}^2}.
\end{equation}
\end{proposition}

The Hankel determinants on the right of \eqref{7.1} and \eqref{7.3} are given 
explicitly below in \eqref{7.10} and \eqref{7.10a}, respectively.
The proofs of both results are based on the following lemma.

\begin{lemma}\label{lem:7.4}
With the sequence $(b_0(x), b_1(x), \ldots)$ as in \eqref{7.0}, we have
for all $n\geq 1$,
\begin{equation}\label{7.5}
H_n\big(b_k(x)\big) 
= (-1)^{n-1}\frac{4p_{n-1}(x)}{n!^2\prod_{\ell=1}^n((x^2-(2\ell-1)^2)}
\cdot H_n\big(E_{2k}(\tfrac{x+1}{2})\big),
\end{equation}
where the polynomial sequence $p_n(x)$ satisfies the recurrence relation
$p_0(x)=1$ and
\begin{equation}\label{7.6}
p_n(x)=\left(x^2-(2n+1)^2\right)p_{n-1}(x) + (-4)^n{n!}^2.
\end{equation}
\end{lemma}

The first few terms of this sequence, after $p_0(x)=1$, are
\begin{align*}
p_1(x) &= x^2-13,\\
p_2(x) &= x^4-38x^2+389,\\
p_3(x) &= x^6-87x^4+2251x^2-21365.
\end{align*}
Before proving Lemma~\ref{lem:7.4}, we derive from it 
Propositions~\ref{prop:7.2} and~\ref{prop:7.3}.

\begin{proof}[Proof of Propositions~\ref{prop:7.2}]
We define the rational functions ${\mathcal K}_n(x)$ by 
${\mathcal K}_0(x)=1$, and for $n\geq 1$ implicitly by
\begin{equation}\label{7.7}
p_n(x)=\left(x^2-3^2\right)\left(x^2-5^2\right)\cdots
\left(x^2-(2n+1)^2\right){\mathcal K}_n(x).
\end{equation}
Then we can rewrite \eqref{7.6} as
\[
{\mathcal K}_n(x) = {\mathcal K}_{n-1}(x)+\frac{\left(2^nn!\right)^2}
{\left(3^2-x^2\right)\left(5^2-x^2\right)\cdots\left((2n+1)^2-x^2\right)}.
\]
Iterating, with ${\mathcal K}_0(x)=1$, we get \eqref{7.2}; then 
\eqref{7.5} with \eqref{7.7} gives \eqref{7.1}.
\end{proof}

\begin{proof}[Proof of Propositions~\ref{prop:7.3}]
This result follows from Propositions~\ref{prop:7.2} with $x=0$. Since by
\eqref{1.5} we have $E_{2k-2}(1/2)=2^{2-2k}E_{2k-2}$, then \eqref{7.0},
\eqref{7.2a}, and Lemma~\ref{lem:4.4} give
\begin{equation}\label{7.8}
H_n\big(b_k(0)\big)=H_n\big((\tfrac{1}{2})^{2k-2}b_k\big)
= 4^{n+1}4^{-n(n+1)}H_n\big(b_k\big)=4^{1-n^2}H_n\big(b_k\big).
\end{equation}
Similarly, we find
\[
H_n\big(E_{2k}(\tfrac{1}{2})\big)=H_n\big((\tfrac{1}{4})^kE_{2k}\big)
=4^{-n-n^2}H_n\big(E_{2k}\big),
\]
and thus, with \eqref{7.8},
\begin{equation}\label{7.9}
\frac{H_n\big(b_k\big)}{H_n\big(E_{2k}\big)} = 4^{-1-n}
\frac{H_n\big(b_k(0)\big)}{H_n\big(E_{2k}(\tfrac{1}{2})\big)}.
\end{equation}
On the other hand, by \eqref{7.2} we have
\[
{\mathcal K}_n={\mathcal K}_n(0) = 1+
\sum_{j=1}^{n-1}\left(\frac{2\cdot 4\cdots(2j)}{3\cdot 5\cdots(2j+1)}\right)^2
=1+\sum_{j=1}^{n-1}\frac{16^j\cdot j!^4}{(2j+1)!^2},
\]
where the right equality is easy to verify. Combining this and \eqref{7.9} with
\eqref{7.1}, we get the desired identity \eqref{7.3}.
\end{proof}

\begin{proof}[Proof of Lemma~\ref{lem:7.4}]
We set $c_k(x):=E_{2k}(\frac{x+1}{2})$ and use the identity \eqref{5.7} in
\cite{DJ1}, namely
\begin{equation}\label{7.10}
H_n(c_k(x)) = (-1)^{\binom{n+1}{2}}\prod_{\ell=1}^n
\left(\frac{\ell^2}{4}\big(x^2-(2\ell-1)^2\big)\right)^{n+1-\ell}.
\end{equation}
Although this is not needed here, we mention that \eqref{7.10} with $x=0$, and
then using \eqref{1.3b} and \eqref{4.8}, yields
\begin{equation}\label{7.10a}
H_n(E_{2k}) = \prod_{\ell=1}^n
\left((2\ell-1)^2(2\ell)^2)\right)^{n+1-\ell};
\end{equation}
see also \cite[Eq.~(3.52)]{Kr1}.
In view of \eqref{7.5} we consider
\begin{align*}
r_n(x):&=\frac{H_n\big(b_k(x)\big)}{H_{n-1}\big(c_k(x)\big)}
=\frac{H_n\big(b_k(x)\big)}{H_n\big(c_k(x)\big)}\cdot
\frac{H_n\big(c_k(x)\big)}{H_{n-1}\big(c_k(x)\big)} \\
&=(-1)^{n-1}\frac{4p_{n-1}(x)}{n!^2\prod_{\ell=1}^n((x^2-(2\ell-1)^2)}\cdot
(-1)^{n-1}\prod_{\ell=1}^n
\left(\frac{\ell^2}{4}\big(x^2-(2\ell-1)^2\big)\right),
\end{align*}
and thus
\begin{equation}\label{7.11}
r_n(x) = -4^{1-n}p_{n-1}(x).
\end{equation}
We can verify by direct computation that \eqref{7.11} holds for $n=1$ and $n=2$.
By Lemma~\ref{lem:2.5} we are then done if we can show that the polynomials
$r_n(x)$ satisfy the recurrence relation
\begin{equation}\label{7.12}
r_{n+1}(x) = -s_n(x)r_n(x)  - t_n(x)r_{n-1}(x),
\end{equation}
where according to \cite[Eq.~(5.5)]{DJ1} we have
\[
s_n(x)=2n^2+n+\frac{1}{4}\left(1-x^2\right),\qquad
t_n(x)= \frac{1}{4}n^2\left((2n-1)^2-x^2\right).
\]
Substituting these terms and \eqref{7.11} into \eqref{7.12}, and then
multiplying both sides by $-4^n$, we get
\begin{equation}\label{7.13}
p_n(x)=\big(x^2-(8n^2+4n+1)\big)p_{n-1}(x)+4n^2\big(x^2-(2n-1)^2\big)p_{n-2}(x).
\end{equation}
On the other hand, by \eqref{7.6} we have
\begin{align}
p_n(x)-\left(x^2-(2n+1)^2\right)p_{n-1}(x) &= (-4)^n{n!}^2,\label{7.14}\\
p_{n-1}(x)-\left(x^2-(2n-1)^2\right)p_{n-2}(x) &= (-4)^{n-1}{(n-1)!}^2.\label{7.15}
\end{align}
Finally, multiplying both sides of \eqref{7.15} by $-4n^2$ and equating the
resulting equation with \eqref{7.14}, then upon simplification we get 
\eqref{7.13}. The proof is now complete.
\end{proof}

\end{document}